\documentclass{amsart}

\usepackage[paperwidth=210mm,paperheight=297mm,hmargin={35mm,35mm},vmargin={35mm,35mm}]{geometry} %margens
\usepackage{amsthm,amssymb,latexsym,amsmath}
\usepackage{mathrsfs}
\usepackage{graphicx}
\usepackage{hyperref}
\input xy
\xyoption{all}
\usepackage[all]{xy}

\newcommand{\fol}{\mbox{${\mathscr F}$}}

\newtheorem{nada}{Nada}[section]
\newtheorem{definition}[nada]{Definition}

\newtheorem{corollary}[nada]{Corollary}
\newtheorem{thm}[nada]{Theorem}
\newtheorem*{thm*}{Theorem}
\newtheorem{lemma}[nada]{Lemma}

\newtheorem{rmk}[nada]{Remark}
\newtheorem{example}[nada]{Example}

\newcommand{\bc}{\begin{center}}
\newcommand{\ec}{\end{center}}
\newcommand{\noi}{\noindent}

\theoremstyle{plain}

\begin{document}

\title{Baum-Bott Residue of Flags of Holomorphic Distributions}
\hyphenation{ho-mo-lo-gi-cal}
\hyphenation{fo-lia-tion}

\begin{abstract}

In this work we extend the residue theory from flag of holomorphic foliations to flag of holomorphic distributions and we provide an effective way to calculate this class in certain cases. As a consequence, we show that if we consider a flag $\mathcal{F} = (\mathcal{F}_{1}, \mathcal{F}_{2})$ of holomorphic distributions on $\mathbb{P}^{3}$, we get a relation between the degrees of the distributions in the flag, the tangency order of distributions, the Euler characteristic and the degree of the curve $C.$

%Let $\mathcal{F} = (\mathcal{F}_{1}, \mathcal{F}_{2})$ be a flag of holomorphic distribution, i. e. the leaves of the distribution $\mathcal{F}_{1}$ are contained of the $\mathcal{F}_{2}$ ones. In [\ref{BCL}], Corr\^ea, Brasselet and Louren\c co showed that it is possible to define the residue class of flag of holomorphic foliations. This residue in general is an homology class. The goal of this paper is to work with flag of distribution. In this context we extend the residue theory to distribution and we appear an effective way of calculate this class in some cases.
 
\end{abstract}

\author{Antonio M. Ferreira}
\address{ Antonio Marcos Ferreira da Silva \\ DMA - UFES,  Rodovia BR-101, Km 60, Bairro Litor\^aneo, S\~ao Mateus-ES, Brazil, CEP 29932-540}
\email{antonio.m.silva@ufes.br}

\author{Fernando Louren\c co}
\address{ Fernando Louren\c co \\ DMM - UFLA,  Campus Universit\'ario, Lavras MG, Brazil, CEP 37200-000}
\email{fernando.lourenco@ufla.br}
\thanks{ }
%\subjclass{Primary 32S65, 37F75; secondary 14F05}
%\keywords{Pfaff Systems,   Residues}

%\dedicatory{}
%\commby{ }
%\begin{abstract}
%\end{abstract}
%\begin{center}
\maketitle
%\end{center}
%\tableofcontents

\section{Introduction}

The residue theory of holomorphic foliations was developed  by several authors, see for example  Baum and Bott [\ref{BB1}, \ref{Baum}] and   Suwa [\ref{Su1}, \ref{Suwa2}]. Let $\mathcal{F}$ be a holomorphic foliation  of dimension $k$ on a complex manifold $M$ and let $\varphi$ be a holomorphic symmetric polynomial of degree $d$ satisfying the vanishing condition $ n-k < d \leq n $ and let $ Z \subset S(\mathcal{F})$ be a compact connect component of singular set of $\mathcal{F}$, there exists a homology class $Res_{\varphi}(\mathcal{F}, Z) \in H_{2(n-d)}(Z; \mathbb{C})$ such that if $M$ is compact

$$ \varphi(\mathcal{N}_{\mathcal{F}})\frown [M] = \sum_{Z} Res_{\varphi}(\mathcal{F}; Z), $$

\noindent where $\mathcal{N}_{\mathcal{F}}$ denotes the normal sheaf of the foliation $\mathcal{F}$.

In general, there is no way to compute the residue class $Res_{\varphi}(\mathcal{F}; Z)$ and this is an open problem in foliation theory. Several authors have worked on this topic and have obtained interesting results, see [\ref{BCL},\ref{BrSu}, \ref{CoLo}, \ref{Dia}, \ref{Vi}].

For instance in the case $k=1$ and $S(\mathcal{F})$ consists only of isolated singularities we have an expression of the above residue, see [\ref{BB1}, Theorem 1]:

$$ Res_{\varphi}(\mathcal{F}; p) = Res_p\left[\begin{array}{cccc} \varphi(Jv)\\ v_1,\ldots,v_n\end{array}\right ],
$$

\noindent where $v=(v_1,\ldots,v_n)$ is a germ of holomorphic vector field at $p$, local representative of $\fol$, $Jv$ is its Jacobian matrix and $Res_p\left[\begin{array}{cccc} \varphi(Jv)\\ v_1,\ldots,v_n\end{array}\right ]$ is the Grothendieck residue.

Let $\mathcal{F}$ be a dimension $k$ holomorphic foliation on  a compact complex manifold $M$ such that $dim S(\mathcal{F})\leq k-1$, and let $Z$ be an irreducible compact component of $S(\mathcal{F})$ of dimension $k-1$. In this case, Baum and Bott in [\ref{BB1}, Theorem 3, p. 285], and  Corr\^ea and Louren\c co in [\ref{CoLo}] give an effective way to compute the residue class $Res_{\varphi}(\mathcal{F}; Z)$, where $\varphi$ is a symmetric homogeneous polynomial  of degree $n-k+1$. To do this,   take a generic point $p \in Z$ such that $p$ is a point where $Z$ is smooth and disjoint from the other singular components. Now, consider $D_{p}$ a ball centered at $p$, of dimension $n-k+1$ sufficiently small and transversal to $Z$ in $p$. Thus

$$Res_{\varphi}(\mathcal{F}; Z) = Res_p\left[\begin{array}{cccc} \varphi(Jv|_{D_{p}})\\ v_1,\ldots,v_{n-k+1}\end{array}\right ][Z]$$
 
\noindent where $Res_p\left[\begin{array}{cccc} \varphi(Jv|_{D_{p}})\\ v_1,\ldots,v_{n-k+1}\end{array}\right ]$ represents the Grothendieck residue of the foliation $\mathcal{F}$ restricted to $D_{p}$. For other progress in residue theory we refer to [\ref{BrSu}, \ref{Dia}].

We define a $2$-flag of holomorphic foliations by a sequence of $2$ foliations $\mathcal{F} = (\mathcal{F}_{1}, \mathcal{F}_{2})$ such that the leaves of the foliation $\mathcal{F}_{1}$ is contained in $\mathcal{F}_{2}$ ones. We also define the singular set of $\mathcal{F}$, by union of the singular sets of the foliations, i.e., $S(\mathcal{F}) := S(\mathcal{F}_{1}) \cup S(\mathcal{F}_{2})$. For an overview about flag theory we refer to [\ref{BCL}].

There exist many works in flags theory. Feigin studied characteristic classes of flags in 1975, see [\ref{Fei}], where the author investigates an obstruction for existence of flags integrably homotopic. Mol in [\ref{Mol}] studied the behavior of singularities of flags and its polar varieties. In the same sense, Corr\^ea and Soares studied the Poincar\' e problem for flags in [\ref{CoSo}]. 

More recently in [\ref{BCL}, Theorem 2] Brasselet, Corr\^ea and Louren\c co studied residues of flag of holomorphic foliations and they proved a residue theorem of Baum-Bott type for flags.

%\begin{thm*}\nonumber(Brasselet, Corr\^ea and Louren\c co) Let $M$ be a complex manifold of dimension $n$ and $E = E_1 \oplus E_2$ be a vector bundle on $M$ such that $E_1$ is a $F_1$-bundle, $E_2$ is a $F_2$-bundle with $F_1\subset F_2\subset TM$ regular foliations. Let $\varphi_{1}$ and $\varphi_2$ be homogeneous symmetric polynomials, of degrees $d_1$ and $d_2$, such that at least one of the inequalities is satisfied.

%\begin{equation}
%\label{000}d_1 > n-rank(F_1 ) \ \ or \ \ d_2 > n-rank(F_2 ) \ \ or \ \ d_1 + d_2 > n-rank(F_1 ).
%\end{equation}

%\noindent Then $\varphi_1(E_1 )\smile \varphi_2(E_2 ) = 0.$

%\end{thm*}

%\begin{thm*} Let  $\mathcal{F} = (\mathcal{F}_{1}, \mathcal{F}_{2})$ be a $2$-flag of holomorphic foliations on a compact complex manifold $M$ of dimension $n$. Let $\varphi_1 , \varphi_2$ be homogeneous symmetric polynomials , respectively of degrees $d_1$ and $d_2$, satisfying the Vanishing Theorem for Flag. Then for each compact connected component $Z$ of $ S(\mathcal{F})$ there exists a class $Res_{\varphi_1, \varphi_2} (\mathcal{F},\mathcal{N}_{\mathcal{F}} , Z) \in H_{2n-2(d_1 +d_2}) (Z;\mathbb{C})$ such that

%$$
%\sum_{\lambda}(i_{\lambda})_{\ast}Res_{\varphi_1, \varphi_2} (\mathcal{F}, \mathcal{N}_{\mathcal{F}}, Z_{\lambda}) = \Big(\varphi_1(\mathcal{N}_{12})\varphi_2(\mathcal{N}_{2})\Big)\frown [M],
%$$

%\noindent where $i_{\lambda}$ denotes the embedding of $Z_{\lambda}$ on $M$.
%\end{thm*}

Although there is a residue theory for flag, it is not simple to calculate the residue of flag in general. The goal of this paper is to show a partial answer of this problem. We start the paper with an extension of the residue theorem of flag to distribution.

\begin{thm}\label{1.0.14} Let $\mathcal{F} = (\mathcal{F}_{1}, \mathcal{F}_{2} )$ be a $2$-flag of singular holomorphic distributions of codimension $s_1$, $s_2$ respectively on a compact complex manifold
$M$ of dimension $n$. Let $\varphi_{1}=c_{p_1}\cdots c_{p_k}$ and $\varphi_{2}=c_{t_1}\cdots c_{t_q}$ be Chern monomials, of degrees $d_1$, $d_2$ respectively, such that $p_i>s_1-s_2$ for some $i$ or $t_j>s_2$ for some $j$. Then for each compact connected component $Z$ of
$S(\mathcal{F})$ there exists a class Res$_{\varphi_{1}, \varphi_{2}} (\mathcal{F} , \mathcal{N}_{\mathcal{F}}; Z )  \in  H_{2n - 2(d_{1} + d_{2} )} (Z; \mathbb{C})$ such that

\begin{equation*}\label{eq.0.2}
\sum_{\lambda}(\iota_{\lambda})_{\ast} \mbox{Res}_{\varphi_{1}, \varphi_{2}} (\mathcal{F} , \mathcal{N}_{\mathcal{F}}; Z_{\lambda} )=
\Big(\varphi_{1}(\mathcal{N}_{12}) \varphi_{2}(\mathcal{N}_{2})\Big) \frown [M]   \ \ \ \mbox{in} \ \ \ H_{2n - 2(d_{1} + d_{2} )} (M; \mathbb{C})
\end{equation*}

\noi where $\iota_{\lambda}$ denotes the embedding of $Z_{\lambda}$ on $M$.
\end{thm}

We prove a result about residue of flag of distributions on isolated singularities.

\begin{thm}\label{1.1} Let $\mathcal{F} = (\mathcal{F}_{1}, \mathcal{F}_{2} )$ be a $2$-flag of holomorphic distributions on a compact complex manifold $M$ of dimension $n$, $\varphi_{1} \ and \ \varphi_{2}$ be homogeneous symmetric polynomials, respectively of degrees $d_{1}>0$ and $d_{2}> 0$ and $p$ be an isolated point of $S(\mathcal{F})$. Then

$$Res_{\varphi_{1}, \varphi_{2}} (\mathcal{F} , \mathcal{N}_{\mathcal{F}}; p ) = 0.$$

\end{thm}

%It is interesting to note that the Theorem \ref{1.1} is not true for Foliations, i.e., the residue of foliations at an isolated singularity is not zero in general.

With this tool on the hand we consider a $2$-flag on $\mathbb{P}^{3}$ and we prove the following.

\begin{thm}\label{1.2} Let $\mathcal{F} = (\mathcal{F}_{1}, \mathcal{F}_{2} )$ be a $2$-flag of holomorphic foliations on $\mathbb{P}^{3}$ with $\deg(\mathcal{F}_{i}) = d_{i}$ thus

\begin{equation}\label{2}
(1 + d_{1} - d_{2})\sum_{Z \in S_{1}(\mathcal{F}_{2})}\deg(Z)Res_{\varphi_{2}}(\mathcal{F}_{2}|_{B_{p}}; p)=\sum_{Z \in S(\mathcal{F})} Res_{c_{1}\varphi_{2}}(\mathcal{F}, \mathcal{N}_{\mathcal{F}}; Z),
\end{equation}

\noindent where $\deg(Z)$ is the degree of the irreducible component $Z$, $Res_{\varphi_{2}}(\mathcal{F}_{2}|_{B_{p}}; p)$ represents the Grothendieck residue of the foliation $\mathcal{F}_{2}|_{B_{p}}$ at $\{p \} = Z \cap B_{p}$  with $B_{p}$ a transversal ball and either $\varphi_{2}=c_{1}^{2}$ or $\varphi_{2}=c_{2}.$

\end{thm} 
 
This previous result is a partial advance in goal of calculates the residue of flags and it is an effective way to calculate this residue when $S(\mathcal{F})$ has only one irreducible component.

For the last result of this paper we need the following definition due to G. N. Costa, see [\ref{Costa}]. Let $\mathcal{F}$ be a holomorphic distribution on $M$ with singular set $\{p_{1}, \dots , p_{r}\} \cup C$, where  $\{p_{1}, \dots , p_{r}\}$ are isolated points and $C$ is an irreducible smooth curve,   and let $\pi : \tilde{M} \longrightarrow M$ be the blow up morphism along $C$ with exceptional divisor $\mathscr{C}=\pi^{-1}({C}$). We say that $\mathcal{F}$ is a {\it special holomorphic distribution} if the pull-back distribution $\tilde{\mathcal{F}}$ on $\tilde{M}$ has only isolated singularities and $\mathscr{C}$ is an invariant set.  We show,  for special distribution, a relation between the degrees of the distributions in the flag, the tangency order of the distributions, the Euler characteristic and the degree of the curve $C$.

\begin{thm}\label{thm 04} Let $\mathcal{F} = (\mathcal{F}_{1}, \mathcal{F}_{2})$ be a $2$-flag of holomorphic distributions on $\mathbb{P}^{3}$ satisfying the conditions:

\begin{enumerate}

 \item $S(\mathcal{F}_{1}) = \{ isolated \ \ points \},$
 
 \item $S(\mathcal{F}_{2}) = \{ isolated \ \ points \} \cup \{ C \}$, where $C$ is an irreducible smooth curve,
 
 \item $\mathcal{F}_{2}$ is special along $C$.

\end{enumerate}

\noindent With this information we have the following relation 
$$ \deg(C)\Big [\Big(1+d_{1}-d_{2}\Big)\Big(-l_{2}(2+3l_{2})+2\Big) + (2l_{2}-l_{1})\Big(-3l_{2}(2+4l_{2}-d_{2})+2 +d_{2} \Big)\Big] = 
$$
$$-\chi(C)\Big(-l_{2}(2+3l_{2})+2\Big)(2l_{2}-l_{1}),$$ 
\noindent where $d_{i} = \deg(\mathcal{F}_{i})$, \ \ $l_{i} = tang(\pi^{\ast}\mathcal{F}, \mathscr{C})$ for blow up $\pi$ along the curve $C$ with exceptional divisor $\mathscr{C}$ and $\chi(C)$ the Euler characteristic of $C$.
\end{thm} 

\begin{corollary}\label{Coro 1.5} If $2l_{2} = l_{1}$ we have
$$  d_{2} = d_{1}+1. $$
\end{corollary}

\begin{corollary}\label{Coro 1.6} If $2l_{2} \neq l_{1}$ we get an expression of the Euler characteristic of the curve $C$.

$$ \chi(C) =  \deg(C) \Big[ \dfrac{1+d_{1}-d_{2}}{2l_{2}-l_{1}}  + \dfrac{2+d_{2} -3l_{2}(2+4l_{2}-d_{2})}{2-l_{2}(2+3l_{2})}\Big].$$

\end{corollary}

%\begin{rmk} Let $\mathcal{F} = (\mathcal{F}_{1}, \mathcal{F}_{2})$ be a $2$-flag of holomorphic distributions on $\mathbb{P}^{3}$ (in this case $\mathcal{F}_{1}$ is a foliation). In general we have

%$$S(\mathcal{F}_{1}) = F_{1} \cup C_{1}$$
%$$S(\mathcal{F}_{2}) = F_{2} \cup C_{2},$$

%\noi where $F_{i}$ is a set of isolated singularities and $C_{i}$ is a set of curves. For simplicity we consider $C_{i}$ just an irreducible smooth curve, for $i=1,2.$

%If $C_{1}$ is not the empty set, we consider the foliation $\mathcal{F}_{1}$ to be special along the curve $C_{1}$. So there is a blow up $\pi : \tilde{\mathbb{P}^{3}} \longrightarrow \mathbb{P}^{3}$ of $\mathbb{P}^{3}$ along $C_{1}$ such that $\tilde{\mathcal{F}} = (\tilde{\mathcal{F}_{1}}, \tilde{\mathcal{F}_{2}})$ is a flag on $\tilde{\mathbb{P}^{3}}$ with
%$$S(\tilde{\mathcal{F}_{1}}) = \tilde{F_{1}}$$
%$$S(\tilde{\mathcal{F}_{2}}) = \tilde{F_{2}} \cup \tilde{C_{2}},$$
%\noi where $\tilde{F_{1}}$ is a set of isolated singularities. For this blow up we reefer [\ref{Costa}]. Therefore the hypothesis 1) of Theorem 1.4 is reasonable and more, it is generic.

%\end{rmk}

%\begin{rmk} We observe that in conditions of Theorem \ref{thm 04} we have $S(\mathcal{F}_{1}) \subset C$ [\ref{OmCoJa1}, Lemma 3.6 p. 12].

%\end{rmk} 

\section{Flag of Holomorphic Distributions}

Let us begin by recall the basic material and results in singular holomorphic foliations and distributions. Let $M$ be a complex manifold of dimension $n$ and $\Theta_{M}$ be the sheaf of germs of holomorphic vector fields. For this section we refer to [\ref{BCL}, \ref{Mol}].

A singular holomorphic distribution of dimension $k$ on $M$ is a coherent subsheaf $\mathcal{F}$ of $\Theta_{M}$ of rank $k$.

If $\mathcal{F}$ satisfies the following integrability condition

$$ [\mathcal{F}_{x} , \mathcal{F}_{x}] \subset \mathcal{F}_{x} \ \ \ \mbox{for} \ \  \mbox{all } x \in M,$$

\noi we say that $\mathcal{F}$ is a holomorphic foliation. The normal sheaf of $\mathcal{F}$ is defined by $\mathcal{N}_{\mathcal{F}} := \Theta_{M}/ \mathcal{F}$, such that it is torsion free ( it means that $\mathcal{F}$ is
saturated).
With this definition we have the following exact sequence 
\begin{center} $0  \longrightarrow \mathcal{F}  \longrightarrow  \Theta_{M}  \longrightarrow \mathcal{N}_{\mathcal{F}} \longrightarrow 0.$
\end{center}

We define the singular set of the distribution $\mathcal{F}$ by  
\begin{center}
    $S(\mathcal{F}) := Sing(\mathcal{N}_{\mathcal{F}} ) = \{p \in M ; \mathcal{N}_{\mathcal{F},p} \  \mbox{is} \ \ \mbox{not} \  \mbox{locally}   \ \mbox{free} \}.$
\end{center}

We assume that $\mbox{codim}(S\mathcal{(F)}) \geq 2.$

In [\ref{BB1}] P. Baum and R. Bott developed a general residue theory for singular holomorphic foliations on $M$ using differential geometry. More precisely, let $\mathcal{F}$ be a holomorphic foliation of dimension $k$ on $M$ and $\varphi$ be a homogeneous symmetric polynomial of degree $d$ satisfying $n-k < d \leq n$. Let $Z$ be a compact connected component of the singular set $S(\mathcal{F})$. Then, there exists a homology class, called residue $Res_{\varphi}( \mathcal{F}; Z) \in H_{2(n-d)} (Z; \mathbb{C})$ such
that it depends only on $\varphi$ and on the local behavior of the leaves of $\mathcal{F}$ near Z, satisfying

\begin{center}
$ \varphi(\mathcal{N}_{\mathcal{F}})\frown [M] = \sum_{Z} Res_{\varphi}( \mathcal{F}; Z). $
\end{center}

In [\ref{Su1}] Suwa developed a residue theory for holomorphic distributions using certain Chern polynomials,
and the residues arise from the vanishing by rank reason, instead of foliations that use the Vanishing Theorem. 

%In general, the residue class $Res(\varphi, \mathcal{F} , Z)$ is just an element in homological group, i.e. we do not have a way to calculates this class. In same particular case, when $k = 1$ and dim$S(\mathcal{F}) =0 $, Baum and Bott in [] showed that the class can be calculated by Grothendieck residue,

%$$ Res(\varphi, \mathcal{F} , Z) =  Res_p\left[\begin{array}{cccc} \varphi(JX)\\ X_1,\ldots,X_n\end{array}\right ],$$

%\noindent where $X = (X_{1}, \dots, X_{n})$ is a germ of holomorphic vector field that induce $\mathcal{F}$ near at $p$.

Now we can define flags of holomorphic distributions. For this consider $M$ a complex manifold of dimension $n$.

\begin{definition}\label{def.2.1} Let $\mathcal{F}_{1},\mathcal{F}_{2}$ be two holomorphic distributions on $M$ of dimensions $q = (q_{1},q_{2})$. We say that $\mathcal{F} := (\mathcal{F}_{1},\mathcal{F}_{2})$ is a $2$-flag of holomorphic distributions if $\mathcal{F}_{1}$ is a coherent
sub $\mathcal{O}_{M}$-module of $\mathcal{F}_{2}.$ Furthermore if each $\mathcal{F}_{i}$ is integrable we say that $\mathcal{F}$ is a $2$-flag of holomorphic foliations.

\end{definition}

We note that, for $ x \in M \setminus \cup_{i = 1}^{2} S (\mathcal{F}_{i})$ the
inclusion relation $T_{x}\mathcal{F}_{1} \subset T_{x}\mathcal{F}_{2}$ holds, giving that the leaves of
$\mathcal{F}_{1}$ are contained in leaves of $\mathcal{F}_{2}$, when we have integrability. Here $T\mathcal{F}_{i}$ is the tangent sheaf of the distribution $\mathcal{F}_{i}$, but throughout the text we will abuse of notation and denote it simply by $\mathcal{F}_{i}$. Now we observe that we have a diagram of short exact sequences of sheaves, called "turtle diagram".

$$ \xymatrix{ 0 \ar[rd]  &  & 0 \ar[ld] &  & 0  \\
   & \mathcal{F}_{1} \ar[rd] \ar[dd] &  & \mathcal{N}_{2} \ar[lu] \ar[ru] &  \\
   &  & \Theta_{M} \ar[ru] \ar[rd]  &  &  \\
   & \mathcal{F}_{2} \ar[ru]   \ar[rd] &  & \mathcal{N}_{1} \ar[uu]  \ar[rd] &  \\
 0 \ar[ru] &  & \mathcal{N}_{12} \ar[ru] \ar[rd]  &  & 0 \\
  & 0 \ar[ru] &   &  0 &  }$$

We define the singular set $S(\mathcal{F})$ of the flag
$\mathcal{F}$ to be the analytic set $S(\mathcal{F}_{1})\cup S(\mathcal{F}_{2})$ and $\mathcal{N}_{\mathcal{F}} := \mathcal{N}_{12} \oplus
\mathcal{N}_{2}$ to be the normal sheaf of the flag, where
$\mathcal{N}_{12}$ is the relative quotient sheaf $ \mathcal{F}_{2} / \mathcal{F}_{1}$.

\section{Chern-Weil Theory of Characteristic Class}

In this section we present the basic tools for working with residue of flags. The residue theory was developed firstly by Baum-Bott by using differential geometry. Lehman and Suwa on the decade of 1980 and 1990 present a new approach of residue theory using Chern-Weil theory. We use this last approach, for more details see [\ref{Suwa2}].

\begin{definition} A connection for a complex vector bundle $E$ on $M$ is a $\mathbb{C}$-linear map
$$ \nabla : A^{0} (M,E) \longrightarrow A^{1} (M, E) $$

\noi that satisfies

$$ \nabla (f.s) = df \otimes s + f. \nabla (s) \ \ \mbox{for} \ \ f \in A^{0}(M) \ \ \mbox{and} \ \ s \in A^{0}(M,E).$$

\end{definition}
If $\nabla$ is a connection for $E$, then it induces a $\mathbb{C}$-linear map
$$ \nabla := \nabla^{2} : A^{1}(M,E) \longrightarrow A^{2}(M,E) $$

\noi satisfying
$$ \nabla (\omega \otimes s) = d\omega \otimes s - \omega \wedge \nabla (s), \ \ \omega \in A^{1}(M), \ \ s \in A^{0} (M,E).$$

\begin{definition} The composition $K := \nabla \circ \nabla : A^{0} (M,E) \longrightarrow A^{2} (M,E)$ is called the curvature of the connection $\nabla$.

\end{definition}

If $\nabla$ denotes a connection for a vector bundle $E$ of rank $r$ and $E$ is trivial on the open set $U$, i.e., $E|_{U} \simeq U \times \mathbb{C}^{r}$ and if $ s = ( s_{1},\ldots,s_{r})$ is a frame of $E$ on $U$, then we can write
$$ \nabla (s_{i}) = \sum_{j = 1}^{r} \theta_{ij} \otimes s_{j} \ \ ; \ \ \theta_{ij} \in A^{1}(U). $$

The connection matrix with respect to $s$ is $\theta = (\theta_{ij})$ . Also, using the curvature definition, we get
$$ K(s_{i}) = \sum_{j = 1}^{r}K_{ij}s_{j}, \ \ \ \mbox{where} \ \ \ K_{ij} = d \theta_{ij} - \sum_{k =1}^{r} \theta_{ik} \wedge \theta_{kj}.$$

The curvature matrix with respect to the frame $s$ is $K = (K_{ij})$. Now, to define the Chern class of a vector bundle $E$, we consider $\sigma_{i}, i = 1,\ldots,r$ the $i$-th elementary symmetric functions in the eigenvalues of the matrix $K$
$$ \det(It + K) = 1 + \sigma_{1}(K)t + \sigma_{2}(K)t^{2} + \cdots + \sigma_{r}(K)t^{r}. $$
Next, we define a $2i$-form of Chern $c_{i}$ on $U$ by
$$ c_{i}(K) := \sigma_{i} \Big(\frac{\sqrt{-1}}{2 \pi}K\Big). $$

In general, if $\varphi$ is a symmetric polynomial  in $r$ variables of degree $d$, we can write $\varphi = \tilde{P}(c_{1},\ldots,c_{r})$ for some polynomial $\tilde{P}$. Then we can define
$$ \varphi(K) := \tilde{P}(c_{1}(K),\ldots,c_{r}(K))$$
\noi which is a closed form on $M$. Therefore, we have a cohomology class of $E$ on $M$, $\varphi (E) := \varphi(K) \in H^{2d}(M; \mathbb{C})$.

\begin{rmk}\label{remark2} Observe that, by a question of rank, if $E$ is  a complex vector bundle of rank $r<n$,
for an arbitrary connection $\nabla$ for $E$  $$ c_p (\nabla)\equiv 0$$ for $r<p\leq n$. 
Thus for a Chern monomial $\varphi =c_{p_1}\cdots c_{p_k}$, if $r<p_i\leq n$ for some $i$, $\varphi(\nabla)\equiv 0$.

\end{rmk}

\begin{rmk}\label{remark1} If $U$ is an open trivializing the vector bundle $E$ and $(s_1,\dots,s_r)$ is a frame for $E$ on $U$, we can define a (local)connection for $E$ on $U$ 
simply doing $\nabla(s_i)=0$ for all i. It is easy to see that the curvature matrix $K$ of this connection is a null matrix, so all Chern class $c_i(E)=0$, for $i>0$ on $H^{2i}(U; \mathbb{C}).$
\end{rmk}

Now let  $E_i$, $i=0,\ldots,q$ be a family of complex vector bundle on a complex manifold $M$, $ \xi$ be the virtual bundle $\displaystyle \xi=\sum _{i=0} ^{q}(-1)^i E_i$ and $\varphi$ be a  homogeneous symmetric polynomial. By definition (\ref{Suwa2}),  we have that $$\varphi(\xi)=\sum _{l} \varphi_l ^{(0)} (E_0)\varphi_l ^{(1)}(E_1)\cdots\varphi_l ^{(q)}(E_q),$$ where $\varphi_l ^{(i)}(E_i)$ is a polynomial in the Chern classes of $E_i$, for each $i$ and $l$.

If $\nabla^{(i)}$ is a connection for $E_i$ consider the family $\nabla^{\bullet}=(\nabla^{(q)},\ldots,\nabla^{(0)}).$ Then $\varphi(\xi)$ is the cohomology class of the differential form 
\begin{equation}\label{eq1} \varphi(\nabla^{\bullet})=\sum _{l} \varphi_l ^{(0)} (\nabla^{(0)} )\wedge\varphi_l ^{(1)}(\nabla^{(1)})\wedge\cdots\wedge\varphi_l ^{(q)}(\nabla^{(q)}).\end{equation}

From above definitions and   Remark \ref{remark1} we get. 

\begin{lemma}\label{lemma 3.4}
Let  $E_i$, $i=1,\ldots,q$ be a family of complex vector bundles on a complex manifold $M$ and $ \xi$ be the virtual bundle $\displaystyle \xi=\sum _{i=0} ^{q}(-1)^i E_i$. Let  $U$ be an open trivializing all $E_i$ and $\varphi$ be a  homogeneous symmetric polynomial of degree $d>0$. Then the   differential form $\varphi(\xi)$ vanishes on $U$.
\end{lemma}

\begin{proof}  In fact, since each $E_i$ is a trivial bundle on $U$, we can take  (see Remark \ref{remark1})  $\nabla^{(i)}$  a connection for $E_i$ on $U$ such that $c_j( \nabla^{(i)})=0$ for $j>0$. Writing $\nabla^{\bullet}=(\nabla^{(q)},\ldots,\nabla^{(0)})$ we get   
\begin{equation}\label{eq1111} \varphi(\nabla^{\bullet})=\sum _{l} \varphi_l ^{(0)} (\nabla^{(0)} )\varphi_l ^{(1)}(\nabla^{(1)})\cdots\varphi_l ^{(q)}(\nabla^{(q)})=0\end{equation} since  $\varphi_l ^{(i)} $ is a polynomial in the Chern classes of $E_i$, for each $i$ and $l$,  and the degree of some $\varphi_l ^i$ is greater than zero.

\end{proof}

\section{Proof of Theorem \ref{1.0.14}}

\begin{proof} We will consider $U$ a relatively compact open neighborhood of
$Z$ on $M$ disjoint from the other components of $S(\mathcal{F})$.We set $U_{0} = U \setminus \{Z\}$ and $U_{1} = U$ and consider the open
covering  $\mathcal{U} = \{U_{0}, U_{1} \}$ of $U$.

We will use the Chern-Weil theory of characteristic classes, see [\ref{Suwa2}] for more details, to compute the Chern class of the normal sheaf $\mathcal{N}_{\mathcal{F}}=\mathcal{N}_{12}\oplus \mathcal{N}_{2}$ of flag. To do this, we take resolutions of the normal sheaves $\mathcal{N}_{12}$ and $\mathcal{N}_{2}$ by real analytic vector bundles $E_{i}^{12}$ and
$E_{j}^{2}$ on $U$ see [\ref{AtHi}].
\begin{equation}
\label{seq1} 0 \longrightarrow A_{U}(E_{q}^{1 2}) \longrightarrow ... \longrightarrow A_{U}(E_{0}^{1 2})
\longrightarrow A_{U} \otimes \mathcal{N}_{1 2} \longrightarrow 0.
\end{equation} 
\begin{equation}
\label{seq2}  0 \longrightarrow A_{U}(E_{r}^{2}) \longrightarrow ... \longrightarrow A_{U}(E_{0}^{2})
\longrightarrow A_{U} \otimes \mathcal{N}_{2} \longrightarrow 0.
\end{equation} 
These sequences are exact on the sheaf level, but on $U_0$ we have exact sequences of vector bundles, then [\ref{BB1}, Definition 4.22 and Lemma 4.17] there exist connections $^{12}\nabla _0 ^{i}$ on $U_0 $ for each $E_{i}^{12}$ and $^{1 2} \nabla _0$ for $N_0 ^{12}$ such that the family of connections $(^{12}\nabla _0 ^{q},\dots, ^{12}\nabla_ 0 ^{0}, ^{12}\nabla _0)$ 
is compatible with  the sequence (\ref{seq1})  and $\varphi_1(^{12} \nabla_0 ^{\bullet})=\varphi_1( ^{1 2} \nabla _0)$, where  $^{12} \nabla_0 ^{\bullet}=(^{12}\nabla _0 ^{q},\dots, ^{12}\nabla_ 0 ^{0})$. 

Analogously, there exist connections $^{2}\nabla _0 ^{i}$ on $U_0 $ for each $E_{i}^{2}$ and $^{2} \nabla _0$ for $N_0 ^{2}$ such that the family of connections $(^{2}\nabla _0 ^{q},\dots, ^{2}\nabla_ 0 ^{0}, ^{2}\nabla _0)$ 
is compatible with the sequence (\ref{seq2}) and $\varphi_2(^{2} \nabla_0 ^{\bullet})=\varphi_2( ^{2} \nabla _0)$, where $^{2} \nabla_0 ^{\bullet}=(^{2}\nabla _0 ^{q},\dots, ^{2}\nabla_ 0 ^{0})$.

Let $^{12}\nabla _1 ^{i}$( respectively $^{2}\nabla _1 ^{i}$) be a connection on $U_1 $ for each $E_{i}^{12}$(respectively
$E_{i}^{2}$)   and set 
$^{12}\nabla _1 ^{\bullet}=(^{12}\nabla _1 ^{q},\dots, ^{12}\nabla _1 ^{0})$ (respectively $^{2}\nabla _1 ^{\bullet}=(^{2}\nabla _1 ^{r},\dots, ^{2}\nabla _1 ^{0})$).

Then the class $\varphi(\mathcal{N}_{\mathcal{F}}) =
\varphi_{1}(\mathcal{N}_{1 2}) \smallsmile
\varphi_{2}(\mathcal{N}_{2})$ in $H^{2(d_{1} + d_{2})}(U; \mathbb{C})$ is represented in  \\ $A^{2(d_{1} +
d_{2})}(U)$  by the cocycle
 $$\sloppy\begin{array}{lllll}\sloppy
\varphi(_{2}^{12}\nabla_{\ast}^{\bullet}) & = \Big(\varphi_{1}(^{12}\nabla_{0}),
\varphi_{1}(^{12}\nabla_{1}^{\bullet})  ,
\varphi_{1}(^{12}\nabla_{0}^{\bullet}, ^{12}\nabla_{1}^{\bullet}) \Big) \\ \\
& \smallsmile \Big(\varphi_{2}(^{2}\nabla_{0}),
\varphi_{2}(^{2}\nabla_{1}^{\bullet})  ,
\varphi_{2}(^{2}\nabla_{0}^{\bullet}, ^{2}\nabla_{1}^{\bullet}) \Big) \\ \\
& = \Big(\varphi_{1}(^{12}\nabla_{0} ) \wedge
\varphi_{2}(^{2}\nabla_{0} ),
\varphi_{1}(^{12}\nabla_{1}^{\bullet}) \wedge
\varphi_{2}(^{2}\nabla_{1}^{\bullet})  ,\varphi_{1}(^{12}\nabla_{0}^{\bullet}) \wedge
\varphi_{2}(^{2}\nabla_{0}^{\bullet}, ^{2}\nabla_{1}^{\bullet}) \\ \\
& +
\varphi_{1}(^{12}\nabla_{0}^{\bullet}, ^{12}\nabla_{1}^{\bullet})
\wedge \varphi_{2}(^{2}\nabla_{1}^{\bullet})\Big).
\end{array}$$

Since the rank of $N_0 ^{12}$ is $s_1-s_2$, the rank of $N^2 _0$ is $s_2$, by hypothesis about the degree of $\varphi_{1}$ and $\varphi_{2}$ and the Remark \ref{remark2} we have $\varphi_{1}(^{12}\nabla_{0} )=0$ or
$\varphi_{2}(^{2}\nabla_{0})=0$. Then

$$\begin{array}{ll}

\varphi(_{2}^{12}\nabla_{\ast}^{\bullet}) & = \Big(0,
\varphi_{1}(^{12}\nabla_{1}^{\bullet}) \wedge
\varphi_{2}(^{2}\nabla_{1}^{\bullet})  ,
\varphi_{1}(^{12}\nabla_{0}^{\bullet}) \wedge
\varphi_{2}(^{2}\nabla_{0}^{\bullet}, ^{2}\nabla_{1}^{\bullet}) \\ \\

& + \varphi_{1}(^{12}\nabla_{0}^{\bullet}, ^{12}\nabla_{1}^{\bullet})
\wedge \varphi_{2}(^{2}\nabla_{1}^{\bullet})\Big).

\end{array}$$

Therefore $\varphi(_{2}^{12}\nabla_{\ast}^{\bullet}) \in A^{2(d_{1} +
d_{2})}(U, U_{0})$. Denoting $[\varphi(_{2}^{12}\nabla_{\ast}^{\bullet})] =
\varphi_{Z}(\mathcal{N}_{\mathcal{F}}, \mathcal{F})$ in \\ $H^{2(d_{1} +
d_{2})}(U, U\backslash Z ; \mathbb{C})$ we have the residue  Res$_{\varphi_{1}, \varphi_{2}}
(\mathcal{N}_{\mathcal{F}}, \mathcal{F}; Z) = A(\varphi_{Z}(\mathcal{N}_{\mathcal{F}}, \mathcal{F}))$ in \\ $H_{2n -
2(d_{1} + d_{2})} (Z ; \mathbb{C}) $, where $A$ is the Alexander homomorphism, see [\ref{ABT}, p. 3028].

\end{proof}

\section{Proof of Theorem \ref{1.1}}

\begin{proof} Since $p$ is an isolated point of $S(\mathcal{F})$, we can take the open $U$, such that, all the vector bundles, $E_i ^{12}, E_j ^2$, on the resolutions of 
$\mathcal{N}_{1 2}$ and $\mathcal{N}_{2} $ (see the sequences (\ref{seq1}), (\ref{seq2}))
are trivial on $U$ and on $U_0=U\setminus \{p\}$. Therefore by the Lemma \ref{lemma 3.4} there exist connections $^{12}\nabla _0 ^{i}$ on $U_0 $ for each $E_{i}^{12}$ such that
if we denote $^{12}\nabla _0 ^{\bullet}$ by $(^{1 2}\nabla _0 ^{(q)},
\dots, ^{1 2}\nabla _ 0 ^{(0)})$ 
$$
\varphi_{1}(^{1 2}\nabla _0 ^{\bullet}) =0.
$$
 Analogously, there exist connections $^{2}\nabla _0 ^{i}$ on $U_0$ for each $E_{i}^{2}$ with the same property. That is, if we denote $^{2}\nabla _ 0 ^{\bullet}$ by $(^{2}\nabla _0 ^{(r)},\dots, ^{2}\nabla_ 0 ^{(0)}). $ Then
$$
\varphi_{2}(^{2}\nabla _0 ^{\bullet}) =0.
$$

Doing the same on $U_1=U$, there exist connections $^{12}\nabla _1 ^{i}$ for each $E_{i}^{12}$ and $^{2}\nabla _1 ^{i}$  for each $E_{i}^{2}$ on $U_1$ such that
if we denote $^{12}\nabla _1 ^{\bullet}$ by $(^{1 2}\nabla _1 ^{(q)},
\dots, ^{1 2}\nabla _1 ^{(0)})$ and  $^{2}\nabla _ 1 ^{\bullet}$ by $(^{2}\nabla _1 ^{(r)},\dots, ^{2}\nabla_ 1 ^{(0)}), $ we have
$$
\varphi_{1}(^{1 2}\nabla _1 ^{\bullet}) =0
\hspace{2cm}
\varphi_{2}(^{2}\nabla _1 ^{\bullet}) =0.
$$
Then the class $\varphi(\mathcal{N}_{\mathcal{F}})$ is
represented in  $A^{2(d_{1} +
d_{2})}(U)$  by the cocycle 
$$
\begin{array}{ll}
\varphi(_{2}^{12}\nabla_{\ast}^{\bullet}) & = \Big(\varphi_{1}(^{12}\nabla_{0}^{\bullet}) \wedge
\varphi_{2}(^{2}\nabla_{0}^{\bullet}),
\varphi_{1}(^{12}\nabla_{1}^{\bullet}) \wedge
\varphi_{2}(^{2}\nabla_{1}^{\bullet})  ,
\varphi_{1}(^{12}\nabla_{0}^{\bullet}) \wedge
\varphi_{2}(^{2}\nabla_{0}^{\bullet}, ^{2}\nabla_{1}^{\bullet}) \\ 
& + \varphi_{1}(^{12}\nabla_{0}^{\bullet}, ^{12}\nabla_{1}^{\bullet})
\wedge \varphi_{2}(^{2}\nabla_{1}^{\bullet})\Big) \\ 
& = \Big( 0,0,0\Big).
\end{array}$$

Therefore 
$$Res_{\varphi_{1}, \varphi_{2}} (\mathcal{F} , \mathcal{N}_{\mathcal{F}}; p ) = 0.$$

\end{proof}

\section{Proof of the Theorem \ref{1.2} }

In this section we will prove the Theorem \ref{1.2} using the Theorem \ref{1.1}. 

%\begin{thm}\label{1.0.14} Let $\mathcal{F} = (\mathcal{F}_{1}, \mathcal{F}_{2} )$ be a 2-flag of holomorphic foliations of
%$\mathbb{P}^{3}$ with $\deg(\mathcal{F}_{i}) = d_{i}$ and singular set of $\mathcal{F}_{1}$, denoted by $Sing(\mathcal{F}_{1})$, contains only isolated singularities. Then

%$$\sum_{Z \in Sing(\mathcal{F})} Res_{c_{1}^{2}c_{1}}(\mathcal{F}, \mathcal{N}_{\mathcal{F}}, Z) = \sum_{Z \in Sing(\mathcal{F})}\partial Z RG_{c_{1}^{2}}(\mathcal{F}_{2}|_{H},Z) (1 + d_{1} - d_{2}), $$

%\noindent where $\partial Z$ be the degree of component $Z$ and $RG_{c_{1}^{2}}(\mathcal{F}_{2}|_{H},Z)$ represents the Grothendieck point residue of the restriction of foliation $\mathcal{F}$ to $H$ at $p = Z \cap H$.
%\end{thm}

\begin{proof} In this prove we will use the transversal disc method used by Baum and Bott in [\ref{BB1}], Visik in [\ref{Vi}] and Corr\^ea and Louren\c co in [\ref{CoLo}].

By Baum-Bott Theorem for flags (Theorem \ref{1.0.14}), we have

\begin{equation}\label{03}
\int_{M} c_{1}(\mathcal{N}_{12})\varphi_{2}(\mathcal{N}_{2}) = \sum_{Z \in S(\mathcal{F})} Res_{c_{1}\varphi_{2}}(\mathcal{F},\mathcal{N}_{\mathcal{F}};Z),
\end{equation}

\noindent where $Res_{c_{1}\varphi_{2}}(\mathcal{F},\mathcal{N}_{\mathcal{F}};Z)$ denotes the flag residue at component $Z$, $\mathcal{N}_{\mathcal{F}} = \mathcal{N}_{12}\oplus \mathcal{N}_{2}$ the normal sheaf and $S(\mathcal{F})$ the singular set of flag $\mathcal{F}.$ 

In this case we consider $M=\mathbb{P}^{3}$ and we have $S(\mathcal{F}) = S_{0}(\mathcal{F}) \cup S_{1}(\mathcal{F})$, where $S_i(\mathcal{F})$ denotes the components of the singular set of the flag $\mathcal{F}$ of pure dimension $i$, for $i=0,1.$

We can rewrite (\ref{03}) as

$$\int_{\mathbb{P}^{3}} c_{1}(\mathcal{N}_{12})\varphi_{2}(\mathcal{N}_{2}) = \sum_{p \in S_{0}(\mathcal{F})} Res_{c_{1}\varphi_{2}}(\mathcal{F},\mathcal{N}_{\mathcal{F}};p) + \sum_{Z \in S_{1}(\mathcal{F})} Res_{c_{1}\varphi_{2}}(\mathcal{F}, \mathcal{N}_{\mathcal{F}};Z).$$

But, by Theorem \ref{1.1} $Res_{c_{1}\varphi_{2}}(\mathcal{F},\mathcal{N}_{\mathcal{F}};p) = 0$, we have

\begin{equation}\label{04}
\int_{\mathbb{P}^{3}} c_{1}(\mathcal{N}_{12})\varphi_{2}(\mathcal{N}_{2}) = \sum_{Z \in S_{1}(\mathcal{F})} Res_{c_{1}\varphi_{2}}(\mathcal{F},\mathcal{N}_{\mathcal{F}};Z).
\end{equation}

%Now we observe the commutative diagram in [\ref{Suwa2}, Proposition 3.11, pg 55] and We will go use this diagram for the foliation $\mathcal{F}_{2}$.

Now we will use the commutative diagram [\ref{Suwa2}, Proposition 3.11, p. 55] for foliation $\mathcal{F}_{2}$

$$ \xymatrix{ H^{2d}(\mathbb{P}^{3}, \mathbb{P}^{3} \backslash Z; \mathbb{C}) \ar[r]  \ar[d]_{A} & H^{2d}(\mathbb{P}^{3} ; \mathbb{C})  \\ H_{2(3-d)}(Z; \mathbb{C}) \ar[r]^{\iota^{\ast}} & H_{2(3-d)}(\mathbb{P}^{3}; \mathbb{C}) \ar[u]_{P}}$$

\noi where $A$ and $P$ denote, respectively, the Alexander homomorphism (isomorphism if $Z$ is nonsingular) and the Poincar\'e homomorphism (isomorphism since $\mathbb{P}^{3}$ is nonsingular), see [\ref{ABT}, p. 3028], $\iota$ denotes the inclusion map of $Z$ in $\mathbb{P}^{3}$ and $\iota^{\ast}$ its induced map in homology group. For go on we use the composition map

$$ \alpha = P \circ \iota^{\ast} : H_{2}(S(\mathcal{F}_{2}); \mathbb{C}) \longrightarrow H^{4}(\mathbb{P}^{3}; \mathbb{C}). $$

By Theorem 1.2 in [\ref{CoLo}] and Baum-Bott Theorem for $\mathcal{F}_{2}$ in [\ref{BB1}, Theorem 1], we have

$$ \sum_{Z \in S_{1}(\mathcal{F}_{2})} \alpha \Big(  Res_{\varphi_{2}}(\mathcal{F}_{2}|_{D};p)[Z]\Big ) = \varphi_{2}(\mathcal{N}_{2}) \ \ \mbox{in} \ \ H^{4}(\mathbb{P}^{3}; \mathbb{C}),$$

\noindent where 

$$Res_{\varphi_{2}}(\mathcal{F}_{2}|_{D};p) = Res_p\left[\begin{array}{cccc} \varphi(Jv|_{D})\\ v_1,v_2\end{array}\right ],
$$

\noi represents the Grothendieck residue of foliation $\mathcal{F}_{2}$ on $D \subset \mathbb{P}^{3}$ at $\{p\}$, with $D$ being a transversal disc to $Z$ such that $D \cap Z = \{p \}$.

Then

\begin{equation}\label{05}
\sum_{Z \in S_{1}(\mathcal{F}_{2})} Res_{\varphi_{2}}(\mathcal{F}_{2}|_{D};p)\eta_{Z} = \varphi_{2}(\mathcal{N}_{2})
\end{equation}

\noindent where $\eta_{Z} = \alpha([Z])$ is the Poincar\'e dual of $[Z]$.

Utilizing the exact sequence

$$ 0\longrightarrow \mathcal{N}_{12} \longrightarrow \mathcal{N}_{1} \longrightarrow \mathcal{N}_{2} \longrightarrow 0,$$

\noindent we have $c_{1}(\mathcal{N}_{12}) = c_{1}(\mathcal{N}_{1}) - c_{1}(\mathcal{N}_{2}).$

But by the exact sequence

$$ 0\longrightarrow \mathcal{F}_{1} \longrightarrow T\mathbb{P}^{3} \longrightarrow \mathcal{N}_{1} \longrightarrow 0,$$

\noi we have $c(T\mathbb{P}^{3})  = c(\mathcal{F}_{1})c(\mathcal{N}_{1})$, and since $\mathcal{F}_{1} = \mathcal{O}_{\mathbb{P}^{3}}(1-d_{1})$ see [\ref{Mol}, p.778],  then \\

$\begin{array}{clc}

c_{1}(\mathcal{N}_{1}) & = c_{1}(T\mathbb{P}^{3}) - c_{1}(\mathcal{F}_{1})   \\ \\
 
 & = 4c_{1}(\mathcal{O}_{\mathbb{P}^{3}}(1)) - (1-d_{1})c_{1}(\mathcal{O}_{\mathbb{P}^{3}}(1)) \\ \\
    
& = (3+d_{1}) h,    
    
\end{array}$ \\

\noi where $\mathcal{O}_{\mathbb{P}^{3}}(1)$ is the line bundle associated to a generic hyperplane $H$ on $\mathbb{P}^{3}$ see [\ref{Daniel}, Definition 2.2.7, p. 69] and $h = c_{1}(\mathcal{O}_{\mathbb{P}^{3}}(1))$ the hyperplane class, see [\ref{GrifHa}, p. 414].
Combining these Chern classes we have

\begin{equation}\label{06}
c_{1}(\mathcal{N}_{12}) = (1 + d_{1}-d_{2})h.
\end{equation}

Replacing the equations (\ref{05}) and (\ref{06}) in equation (\ref{04}), we have,

$$\int_{\mathbb{P}^{3}} (1+d_{1}-d_{2})h \sum_{Z \in S_{1}(\mathcal{F}_{2})} Res_{\varphi_{2}}(\mathcal{F}_{2}|_{D};p)\eta_{Z} = \sum_{Z \in S_{1}(\mathcal{F})} Res_{c_{1}\varphi_{2}}(\mathcal{F},\mathcal{N}_{\mathcal{F}};Z).
$$

$$(1+d_{1}-d_{2})\sum_{Z \in S_{1}(\mathcal{F}_{2})}Res_{\varphi_{2}}(\mathcal{F}_{2}|_{D};p)\int_{\mathbb{P}^{3}} h \eta_{Z} = \sum_{Z \in S_{1}(\mathcal{F})} Res_{c_{1}\varphi_{2}}(\mathcal{F},\mathcal{N}_{\mathcal{F}};Z).
$$ 
 
We note that  $\displaystyle\int_{\mathbb{P}^{3}} h \eta_{Z} = \deg(Z)$. Therefore

$$(1+d_{1}-d_{2})\sum_{Z \in S_{1}(\mathcal{F}_{2})}Res_{\varphi_{2}}(\mathcal{F}_{2}|_{D};p)\deg(Z)= \sum_{Z \in S(\mathcal{F})} Res_{c_{1}\varphi_{2}}(\mathcal{F},\mathcal{N}_{\mathcal{F}};Z).
$$

\end{proof}

\section{Proof of the Theorem \ref{thm 04}}

In this section we prove the Theorem \ref{thm 04} and we observe that its consequences in Corollary \ref{Coro 1.5} and Corollary \ref{Coro 1.6} are immediate.

\begin{proof} We consider $\pi : \tilde{\mathbb{P}^{3}} \rightarrow \mathbb{P}^{3}$ the blow up of $\mathbb{P}^{3}$ along the curve $C$ with exceptional divisor $\mathscr{C}$. We obtain a flag $\tilde{\mathcal{F}} = (\tilde{\mathcal{F}_{1}}, \tilde{\mathcal{F}_{2}})$ on $\tilde{\mathbb{P}^{3}}$ which has only isolated singularities by hypothesis that $\mathcal{F}_{2}$ is special along $C$. For this flag we have see ([\ref{No}], Lemma 2.2 (ii) p. 889)

\begin{equation}\label{eq 1.1}
\left\{\begin{array}{cc}
\tilde{\mathcal{F}_{1}} = & \pi^{\ast}\mathcal{F}_{1}\otimes [\mathscr{C}]^{l_{1}} \\
\tilde{\mathcal{F}_{2}} = & \pi^{\ast}\mathcal{F}_{2}\otimes [\mathscr{C}]^{l_{2}}. \\
\end{array}\right.
\end{equation}

So by Theorem \ref{1.0.14} and Theorem \ref{1.1} we have

\begin{equation}\label{eq 1.2}
\int_{\tilde{\mathbb{P}^{3}}} c_{1}(\tilde{\mathcal{N}_{12}})c_{2}(\tilde{\mathcal{N}_{2}}) = 0.
\end{equation}

Now we finish the prove exploring the equation (\ref{eq 1.2}). Before we present the necessary data for this. From the expression (\ref{eq 1.1}) we get

\begin{equation}\label{eq 1.3}
c_{1} (\tilde{\mathcal{F}_{1}}) = (1-d_{1})\pi^{\ast}h + l_{1}\mathscr{C} \ \ \ \mbox{and} \ \ \ c_{1} (\tilde{\mathcal{F}_{2}}) = (2-d_{2})\pi^{\ast}h + 2l_{2}\mathscr{C},
\end{equation}

\noindent where $h$ is the hyperplane class $c_{1}(\mathcal{O}_{\mathbb{P}^{3}}(1))$ and by abusing of notation we consider $c_{1}([\mathscr{C}]) = \mathscr{C}.$

Now by relation (\ref{eq 1.3}) and the short exact sequence on $\tilde{\mathbb{P}}^{3}$,
$$ 0 \rightarrow \tilde{\mathcal{F}_{1}} \rightarrow \tilde{\mathcal{F}_{2}} \rightarrow \tilde{\mathcal{N}_{12}} \rightarrow 0 $$
\noi we have $c_{1}(\tilde{\mathcal{N}}_{12}) = (1+d_{1}-d_{2})\pi^{\ast}h + (2l_{2}-l_{1})\mathscr{C}$. 

And by the short exact sequence on $\tilde{\mathbb{P}^{3}}$
$$ 0 \rightarrow \tilde{\mathcal{F}_{2}} \rightarrow \tilde{\mathbb{P}^{3}} \rightarrow \tilde{\mathcal{N}_{2}} \rightarrow 0 $$

\noi we get 
$$c_{2}(\tilde{\mathcal{N}_{2}}) = c_{2}(\tilde{\mathbb{P}_{3}}) - c_{2}(\tilde{\mathcal{F}_{2}}) -c_{1}(\tilde{\mathcal{F}_{2}})c_{1}(\tilde{\mathcal{N}_{2}}),$$

\noi where
$$c_{1}(\tilde{\mathcal{N}_{2}}) = (2+d_{2})\pi^{\ast}h -(2l_{2}+l_{1})\mathscr{C},$$

$$c_{2}(\tilde{\mathcal{F}_{2}}) = \pi^{\ast}c_{2}(\mathcal{F}_{2}) + c_{1}(\pi^{\ast}\mathcal{F}_{2})c_{1}(\mathscr{C}^{l_{2}}) +l_{2}^{2}\mathscr{C}^{2}.$$

From Theorem 3.1 p. 14 in [\ref{OmCoJa1}] we have $c_{2}(\mathcal{F}_{2}) = \Big(2+d_{2}^{2} - \deg(C)\Big)h^{2}$ so 

$$c_{2}(\tilde{\mathcal{F}_{2}}) = \Big(2+d_{2}^{2}- \deg C\Big)\pi^{\ast}h^{2} + l_{2}(2-d_{2})\pi^{\ast}h\mathscr{C} +l_{2}^{2}\mathscr{C}^{2}.$$ 

From Porteous Theorem [\ref{Por}, Theorem 2 p. 123], 
$$c_{2}(\tilde{\mathbb{P}^{3}}) = 6\pi^{\ast}h^{2} - \mathscr{C}^{2} -\pi_{\mathscr{C}}^{\ast}c_{1}(TC)\mathscr{C}.$$ 
With this at hand we have
$$c_{2}(\tilde{\mathcal{N}_{2}}) = \deg (C) \pi^{\ast}h^{2} + \Big(-3l_{2}d_{2}-2l_{2}-d_{2}+2\Big)\pi^{\ast}h\mathscr{C} - \pi^{\ast}_{\mathscr{C}}c_{1}(TC)\mathscr{C} + \Big(3l_{2}^{2}+2l_{2}-1\Big)\mathscr{C}^{2}.$$

So we can calculate the product of Chern classes 
$$
\begin{array}{ll}
c_{1}(\tilde{\mathcal{N}_{12}})c_{2}(\tilde{\mathcal{N}_{2}}) & = (1+d_{1}-d_{2})\deg(C)\pi^{\ast}h^{3} + \\ \\  &(1+d_{1}-d_{2})(-3l_{2}d_{2}-2l_{2}-d_{2}+2)\pi^{\ast}h^{2}\mathscr{C} \\ \\
& = - (1+d_{1}-d_{2})\pi^{\ast}_{\mathscr{C}}c_{1}(TC)\mathscr{C}\pi^{\ast}h + (1+d_{1}-d_{2})(3l_{2}^{2}+2l_{2}-1)\pi^{\ast}h\mathscr{C}^{2} \\ \\
& + (2l_{2}-l_{1})\deg(C)\pi^{\ast}h^{2}\mathscr{C} + (2l_{2}-l_{1})(-3l_{2}d_{2}-2l_{2}-d_{2}+2)\pi^{\ast}h\mathscr{C}^{2} \\ \\
& - (2l_{2}-l_{1})\pi^{\ast}_{\mathscr{C}}c_{1}(TC)\mathscr{C}^{2} + (2l_{2}-l_{1})(3l_{2}^{2}+2l_{2}-1)\mathscr{C}^{3}.
\end{array}$$

Now we use the properties of intersection theory see [\ref{No}].

\begin{enumerate}

 \item $\displaystyle\int_{\tilde{\mathbb{P}^{3}}}\pi^{\ast}h^{3} = 1.$

 \item $\displaystyle\int_{\tilde{\mathbb{P}^{3}}}\pi^{\ast}h^{2}\mathscr{C} = \int_{\mathscr{C}}\pi^{\ast}h^{2} =  \int_{C}h^{2}=0 .$ 

 \item $\displaystyle\int_{\tilde{\mathbb{P}^{3}}}\pi^{\ast}h\mathscr{C}^{2} = \int_{\mathscr{C}}\pi^{\ast}h\mathscr{C} =  -\int_{C}h=-\deg(C) .$

 \item $\displaystyle\int_{\tilde{\mathbb{P}^{3}}}\mathscr{C}^{3} = \int_{\mathscr{C}}\mathscr{C}^{2} =  \chi(C) - 4\deg(C) .$
 
 \item $\displaystyle\int_{\tilde{\mathbb{P}^{3}}}\pi^{\ast}h\pi^{\ast}_{\mathscr{C}}c_{1}(TC)\mathscr{C} = 0 .$

 \item $\displaystyle\int_{\tilde{\mathbb{P}^{3}}}\pi^{\ast}_{\mathscr{C}}c_{1}(TC)\mathscr{C}^{2} = \int_{C}c_{1}(TC) = \chi(C) .$ 
\end{enumerate}

\noi Integrating the Chern class $c_{1}(\tilde{\mathcal{N}_{12}})c_{2}(\tilde{\mathcal{N}_{2}})$ on $\tilde{\mathbb{P}_{3}}$  we get

$$
\begin{array}{ll}

\displaystyle\int_{\tilde{\mathbb{P}_{3}}}  c_{1}(\tilde{\mathcal{N}_{12}})c_{2}(\tilde{\mathcal{N}_{2}}) & = (1+d_{1}-d_{2})\deg(C) + (1+d_{1}-d_{2})(3l_{2}^{2}+2l_{2}-1)(-\deg(C)) \\ \\

& + (2l_{2}-l_{1})(-3l_{2}d_{2}-2l_{2}-d_{2}+2)(-\deg(C))-(2l_{2}-l_{1})\chi(C) \\ \\

& + (2l_{2}-l_{1})(3l_{2}^{2}+2l_{2}-1)(\chi(C)-4\deg(C)) \\ \\

& = \deg(C)\Big[(1+d_{1}-d_{2})(-l_{2}(2+3l_{2})+2) + \\ \\ &(2l_{2}-l_{1})(-3l_{2}(2+4l_{2}-d_{2})+2+d_{2})  \Big] \\ \\

& + \chi(C)(2l_{2}-l_{1})\Big(l_{2}(2+3l_{2})-2\Big) \\ \\

& = 0.

\end{array}$$

\noindent Therefore, 

\noi $\deg(C)\Big[(1+d_{1}-d_{2})(-l_{2}(2+3l_{2})+2) + (2l_{2}-l_{1})(-3l_{2}(2+4l_{2}-d_{2})+2+d_{2})  \Big]  \\ =\chi(C)(2l_{2}-l_{1})\Big(-l_{2}(2+3l_{2})+2\Big).$
\end{proof}

\section{Examples}

This section is dedicated to show examples to illustrate some results.

\begin{example} Let $\mathcal{F} = (\mathcal{F}_{1}, \mathcal{F}_{2})$ be a $2$-flag on $\mathbb{P}^{3}$ where $\mathcal{F}_{2}$ is the codimension one holomorphic foliation induced by homogeneous $1$-form

$$\omega = -z_{0}z_{3}dz_{0} -z_{1}z_{3}dz_{1}-z_{2}z_{3}dz_{2} + (z_{0}^2 + z_{1}^{2} + z_{2}^{2})dz_{3}. $$

The foliation $\mathcal{F}_{1}$ is induced by homogeneous vector field

$$ X = z_{1}z_{3}\dfrac{\partial}{\partial z_{0}} -z_{0}z_{3}\dfrac{\partial}{\partial z_{1}} + (z_{0}^2 + z_{1}^{2} + z_{2}^{2})\dfrac{\partial}{\partial z_{2}} + z_{2}z_{3}\dfrac{\partial}{\partial z_{3}}. $$

The singular set of $\mathcal{F}_{2}$ is given by 

$$ S(\mathcal{F}_{2}) = \Big\{ C=\{z_{3}=z_{0}^{2}+z_{1}^{2}+z_{2}^{2} = 0\}, P=[0:0:0:1]       \Big\}.$$

By Theorem \ref{1.2} we have the equality
$$\displaystyle\sum_{Z \in S(\mathcal{F})} Res_{c_{1}c_{1}^{2}}(\mathcal{F}, \mathcal{N}_{\mathcal{F}}; Z) =(1 + d_{1} - d_{2})\deg(C)Res_{c_{1}^{2}}(\mathcal{F}_{2}|_{D}; q),$$

\noi where $ C \cap D = \{q\}$ with $D$ a transversal disc to $C$.

Now we will calculate the sum of residues of flag using the above expression. For this let us consider the chart $U_{0} = \{z_{0} =1 \}$, with coordinates $x =z_{1}/z_{0}, y =z_{2}/z_{0}$ and $z =z_{3}/z_{0}$, so

%and the notations of the local coordinates $x = z_{1}/z_{0}, y = z_{2}/z_{0}$ and $z = z_{3}/z_{0}.$

%$$\omega|_{U_{0}} = -xzdx - yzdy + (x^{2} + y^{2} +1)dz $$

$$C|_{U_{0}} = \{ z= 1+x^{2}+y^{2} = 0\}.$$

Let us consider a small transversal disc $D \subset \{x=0\}$ such that $ D \cap C =\{ (0,\sqrt{-1},0) \} = \{q\}.$

$$\omega|_{D} = -yzdy + (1+y^{2})dz.$$

The dual vector field of this $1$-form is given by

$$Y_{\omega} = (1+y^{2})\dfrac{\partial}{\partial y} + (yz)\dfrac{\partial}{\partial z}.$$

\noi The Jacabian matrix of the $Y_{\omega}$ is

$$
JY_{\omega}=\left[\begin{array}{cc}
2y & 0 \\
z &  y
\end{array}\right].$$

Therefore, we have the residue

$$Res_{c_{1}^{2}}(\mathcal{F}_{2}|_{D};q)= \dfrac{c_{1}^{2}(JY_{\omega}(q))}{\det(JY_{\omega}(q))}= \dfrac{9}{2}.$$

Since $d_{1} = 2, d_{2} = 1$ and $deg(C)=2$ we have

$$\displaystyle\sum_{Z \in S(\mathcal{F})} Res_{c_{1}c_{1}^{2}}(\mathcal{F}, \mathcal{N}_{\mathcal{F}}; Z) =18.$$

\end{example}

\begin{example} We consider now the foliation $\mathcal{G}$ on $\mathbb{P}^{2}$ induced by 1-form 
$$\eta = \Big[(z_{0}-z_{1})z_{1} + z_{2}(z_{0}-z_{2})\Big]dz_{0} + z_{0}(z_{1}-z_{0})dz_{1} + z_{0}(z_{2}-z_{0})dz_{2}.$$
This foliation has singular set $$S(\mathcal{G}) = \Big\{ p_{1} = [0:1:\sqrt{-1}], p_{2} = [0:1:-\sqrt{-1}], p_{3} = [1:1:1] \Big\}.$$
Now let $\pi$ be the rational map
$$\begin{array}{cccc}
\pi : & \mathbb{P}^{3} & \longrightarrow  &  \mathbb{P}^{2} \\
& [z_{0}:z_{1}:z_{2}:z_{3}] & \longmapsto & [z_{0}:z_{1}:z_{2}].
\end{array}$$
If we consider the pull back of $\mathcal{G}$ by $\pi$ we have the holomorphic foliation $\mathcal{F}_ {2} = \pi^{\ast}\mathcal{G}$ on $\mathbb{P}^{3}$ of degree one ($d_{2} = 1$) given by $1$-form
$$\omega = \Big[(z_{0}-z_{1})z_{1} + z_{2}(z_{0}-z_{2})\Big]dz_{0} + z_{0}(z_{1}-z_{0})dz_{1} + z_{0}(z_{2}-z_{0})dz_{2}$$

\noi which its singular set is the union of three lines

$$ S(\mathcal{F}_{2}) = L_{1} \cup L_{2} \cup L_{3},$$
where 
$L_{1} = [0:z_{1}:\sqrt{-1}z_{1}:z_{3}],$
$L_{2} = [0:z_{1}:-\sqrt{-1}z_{1}:z_{3}],$
$L_{3} = [z_{1}:z_{1}:z_{1}:z_{3}].$ 

We set the holomorphic one dimensional foliation $\mathcal{F}_{1}$ on $\mathbb{P}^{3}$ of degree two ($d_{1} = 2$) induced by homogeneous vector field
$$ X =z_{0}( z_{0} - z_{1}) \dfrac{\partial}{\partial z_{0}} + \Big[(z_{0}-z_{1})z_{1} + z_{2}(z_{0}-z_{2})\Big]\dfrac{\partial}{\partial z_{1}}  .$$
We observe that the foliation $\mathcal{F}_{1}$ is a subfoliation of $\mathcal{F}_{2}$ because $\omega(X) = 0.$ Furthermore the flag $\mathcal{F} = (\mathcal{F}_{1}, \mathcal{F}_{2})$ satisfies the hypothesis of Theorem \ref{1.2}. Thus we can calculate the sum of residues of the flag.
$$
\begin{array}{ll}
\displaystyle\sum_{Z \in S(\mathcal{F})} Res_{c_{1}c_{1}^{2}}(\mathcal{F}, \mathcal{N}_{\mathcal{F}}; Z) & = (1+d_{1}-d_{2})\Big[\deg(L_{1})Res_{c_{1}^{2}}(\mathcal{F}_{2}|_{D};t_{1})+ \\ 
& + \deg(L_{2})Res_{c_{1}^{2}}(\mathcal{F}_{2}|_{D};t_{2}) + \deg(L_{3})Res_{c_{1}^{2}}(\mathcal{F}_{2}|_{D};t_{3})\Big].
\end{array}$$

On chart $U_{3} = \{ z_{3} \neq 0 \}$ we have coordinates $x =z_{0}/z_{3}, y =z_{1}/z_{3}$ and $z =z_{2}/z_{3}$ with $l_{1} = L_{1}|_{U_{3}} = (0,y,\sqrt{-1}y), l_{2} = L_{2}|_{U_{3}} = (0,y,-\sqrt{-1}y)$ and $l_{3} = L_{3}|_{U_{3}} = (y,y,y).$ Furthermore $\omega$ on $U_{3}$ is given by
$$ \omega|_{U_{3}} = \Big[(x-y)y + z(x-z)\Big]dx + x(y-x)dy + x(z-x)dz.$$
Now we choose a small transversal disc at each line
 $D = \{ z=1 \}.$
The foliation $\mathcal{F}_{2}$ is given on $D$ by $1$-form
$$\omega|_{D} = \Big[(x-y)y + (x-1)\Big]dx + x(y-x)dy$$
\noi and its dual vector field is
$$ Y_{\omega} = x(y-x)\dfrac{\partial}{\partial x} - \Big[(x-y)y + (x-1)\Big]\dfrac{\partial}{\partial y}.$$
In particular, we have the Jacabian matrix of dual vector field $JY_{\omega}$
$$
JY_{\omega}=\left[\begin{array}{cc}
y-2x & x \\
-y-1 &  -x+2y
\end{array}\right].$$ Thus 
$$Res_{c_{1}^{2}}(\mathcal{F}_{2}|_{D};t_{1}) = \dfrac{c_{1}^{2}(JY_{\omega})(t_{1})}{\det(JY_{\omega})(t_{1})} = \dfrac{9}{2},$$
$$Res_{c_{1}^{2}}(\mathcal{F}_{2}|_{D};t_{2}) = \dfrac{c_{1}^{2}(JY_{\omega})(t_{2})}{\det(JY_{\omega})(t_{2})} = \dfrac{9}{2},$$
$$Res_{c_{1}^{2}}(\mathcal{F}_{2}|_{D};t_{3}) = \dfrac{c_{1}^{2}(JY_{\omega})(t_{3})}{\det(JY_{\omega})(t_{3})} = 0,$$

\noi where $t_{i}$ is the intersect point $D \cap l_{i}.$

Therefore
$$\sum_{Z \in S(\mathcal{F})} Res_{c_{1}c_{1}^{2}}(\mathcal{F}, \mathcal{N}_{\mathcal{F}}; Z) = 2(\dfrac{9}{2} + \dfrac{9}{2} + 0) = 18,$$

\noi since $\deg(l_{i}) = 1$ for $i = 1,2,3$ and $1+d_{1}-d_{2} = 2$.

To finish the example we confirm the residue calculation using Theorem \ref{1.0.14}. In this case, one has

$$\int_{\mathbb{P}^{3}} c_{1}(\mathcal{N}_{12})c_{1}^{2}(\mathcal{N}_{2}) = \int_{\mathbb{P}^{3}} (1+d_{1}-d_{2})h (2+d_{2})^{2}h^{2} = 18.$$

\end{example}

\subsection*{Acknowledgments}
We are grateful to Mauricio Corr\^ea and  Marcio G. Soares    for interesting conversations. The authors were  partially supported by the FAPEMIG [grant number 38155289/2021].

\end{document}